\documentclass[12pt, reqno]{amsart}

\usepackage[authoryear]{natbib}
\setcitestyle{aysep={}} 

\usepackage[breaklinks=true, hidelinks]{hyperref}
\usepackage{breakcites}

\usepackage[foot]{amsaddr}
\usepackage{amsmath}
\usepackage{amssymb}
\usepackage{placeins}
\usepackage{amsfonts}
\usepackage{verbatim}
\usepackage{float}
\usepackage{xcolor}
\usepackage{tikz}
\usepackage{hyperref}
\usepackage{subcaption}

\makeindex

 \newtheorem{theorem}{Theorem}[section]

\theoremstyle{definition}

\theoremstyle{remark}

\newtheorem{fact*}{Fact}

\DeclareMathOperator{\Tr}{Tr}

\newcommand{\abs}[1]{\left\vert#1\right\vert}

\newcommand{\til}{\raise.17ex\hbox{$\scriptstyle\mathtt{\sim}$}}

\newcommand\beq{\begin{equation}}

\newcommand\eeq{\end{equation}}

\newcommand{\bbm}{\left[ \begin{smallmatrix}}
\newcommand{\ebm}{\end{smallmatrix} \right]}
\newcommand{\bpm}{\left( \begin{smallmatrix}}
\newcommand{\epm}{\end{smallmatrix} \right)}
\numberwithin{equation}{section}

\newlength{\Mheight}
\newlength{\cwidth}

\usepackage[colorinlistoftodos]{todonotes}

\definecolor{plummet}{RGB}{142, 69, 133}

\title[Freeness-of-trade]{The geometry of inconvenience and perverse equilibria in trade networks}

\author{
Michael Coopman$^{a}$
}
\author{
Austin Jacobs$^{a}$
}

\author[]{Henry Pascoe$^{b}$
}

\author[]{
J. E. Pascoe$^{c}$
}

\thanks{\\
$^a$ University of Florida, Department of Mathematics, Gainesville, Florida
\\
$^{b}$ IE University, School of Politics, Economics, and Global Affairs, Madrid, Spain. Corresponding author email: henry.pascoe@ie.edu
\\
$^{c}$ Drexel University, Department of Mathematics, Philadelphia, Pennsylvania.  
\\
\vspace{.1cm}
\\
Partially supported by National Science Foundation Mathematical
Science Postdoctoral Research Fellowship  
DMS 1606260 and NSF Analysis grant DMS 1953963 and 2319010 }

\date{}

\begin{document}

\begin{abstract} %
The structure bilateral trading costs is one of the key features of international trade.   Drawing upon the freeness-of-trade matrix, which allows the modeling of N-state trade costs, we develop a ``geometry of inconvenience'' to better understand how they impact equilbrium outcomes. The freeness-of-trade matrix was introduced in a model by Mossay and Tabuchi, where they essentially proved that if a freeness-of-trade matrix is positive definite, then the corresponding model admits a unique equilibrium. Drawing upon the spectral theory of metrics, we prove the model admits nonunique, perverse, equilibria. We use this result to provide a family of policy relevant bipartite examples, with substantive applications to economic sanctions. More generally, we show how the network structure of the freeness of trade is central to understanding the impacts of  policy interventions.
\end{abstract}

\maketitle
\section{Introduction}

Interstate trade is costly. In this paper we consider how the structure of the bilateral costs of trade, including cases where direct trade is prohibitively costly or impossible due to sanctions, geography, or other factors, impact international trade. To do so, we consider imperfect competition models of international trade in the presence of price effects, and characterize the impact of the structure of trading costs in such models upon equilibrium existence and outcomes.  

Mossay and Tabuchi present a three country model which provides a useful starting point for our analysis because, unlike many models of trade, it does not abstract away from such trade costs by assuming factor price equalization. We consider when their results generalize to models with more than three countries, and the consequences for welfare and other key outcomes of interest when they fail to do so.  In particular, we find that partitioned trade networks of the sort we characterize below may result in multiple, perverse, rent-seeking equilibria.  

  Mossay and Tabuchi’s model \citeyearpar{MT15} does not assume factor price equalization, allowing them to consider the size, neighboring, price and integration effects of market liberalization. The model is constructed following Krugman’s three country model \citeyearpar{krugman1980}, like much other scholarship in this area \citep{chaney_2008, ossa_2011, venables_1987}. The Mossay and Tabuchi model is notable for its ability to allow for country size differences and trade cost heterogeneity. Thus, it is a good fit for our central interest, the ease of trade between countries, rather than productivity or factor endowment differences as in the Ricardo and Hecksher-Ohlin models, respectively. 

  In their analysis they show that general equilibrium exist and are unique. They use these results to consider the impact of preferential trade agreements on third states. A major limitation of \cite{MT15} is that it is only a three-country model. They speculate that their results may generalize to an N-country model, providing some limited extension under trade symmetry or country size symmetry. As we note below, it is elementary that their results generalize to four countries. However, we characterize conditions under which their results generalize to models of $N>4$, and comment on the implications for when they do not. By doing so, we identify a “geometry of inconvenience” in trade networks of more than four countries, with implications for international sanctions, among other potential applications. 

In particular, we identify a family of trade networks in which the unique equilibrium identified by \cite{MT15} does not exist, and instead three perverse equilibria exist.\footnote{See \cite{fujita1999}, \cite{mercenier1995}, and \cite{venables_1984}  for discussion of multiple stable equilibrium in models of international trade.} Interestingly, only specific trade configurations result in such equilibria. 
Specifically, we discuss those characterized by ``anti-blocs,'' groups of states that for some reason, such as politics or geography, have no or very little trade, in particular configurations are susceptible to perverse equilibria. We use these results to comment on the impact of economic sanctions on wages, welfare, and prices. We also consider the impact of other unilateral and multilateral policy shifts. 

For instance, the current sanctions by the United States, European Union, and others on Russia have caused a renewed interest in sanctions busting and its consequences. The United Arab Emirates, China, Kazakhstan, Turkey and others are all suspected of helping to bust sanctions by serving as intermediaries between the senders of sanctions and Russia.  This raises the question of how such sanctions busting impacts international trade, and what policy responses may help improve outcomes. In the analysis that follows, we propose a geometric approach to understanding such networks, and engage in policy analysis to address such pressing questions. One key implication is that secondary sanctions, to the extent they isolate sanctions-busting states from each other and the global economy more generally, may be counterproductive, simply increasing opportunities for rent seeking and other undesirable outcomes. Instead, fostering new, more favorable trading options for potential sanctions busters may be a more fruitful approach. 

In sum, all sanctions regimes, even those with similar opportunity cost of forgone direct trade between sender and target, are not created equal. Considering the geometry of inconvenience created by sanctions allows us to comment on their impact, not only on sending and target states, but also third-party states. Furthermore, these results have important implications for attitudes toward globalization. Globalization backlash, rather than a reaction to distributional consequences and imperfect redistribution \citep{milner_1999}, may instead be a reaction to the ill effects of more extensive, cheaper, trade in partitioned systems. The increased use of economics sanctions may, therefore, be connected to rising backlash against globalization.

We next introduce Mossay and Tabuchi’s model, generalized to N countries, along with some mathematical preliminaries. Building on this discussion, we characterize the conditions under which unique equilibrium do not exist for models with more than 4 countries. We then discuss implications of our results for understanding economic sanctions. We next provide policy implications via a series of numerical examples. The final section concludes with a discussion of avenues for future research made possible by our approach.  

\section{The Model}

The model generalizes Mossay and Tabuchi's three country model to arbitrarily many countries, $N$, with a manufacturing sector producing a differentiated good. Here we introduce the primitives of the model before proceeding. Let $L_i$ denote the the mass of immobile workers in country $i$, with $\sum L_i =1$. The utility of an individual country is given by Dixit-Stiglitz preferences,

\begin{equation}
   U_i=  \left[\sum_{j=1}^N \int_{v\in V_j} q_{ji}(v)^{\frac{\sigma-1}{\sigma}}dv\right]^{\frac{\sigma}{\sigma-1}}
\end{equation}

where $q_{ij}(v)$ is the amount of variety $v$ produced in country $j$ and consumed in country $i$, $V_j$ is the set of varieties produced in country $j$ and $\sigma>1$ is the elasticity of substitution between any two varieties. A worker in country $i$ earning wage $w_i$ has the following budget constraint,

\begin{equation}
    \sum_j\int_{v\in V_j} p_{ji}(v)q_{ji}(v)dv = {w_i}
\end{equation}

where $p_{ji}(v)$ is the delivery price of variety $v$ produced in country $j$ and consumed in country $i$. We follow Mossay and Tabuchi in dropping the variety label $v$, and thus tacitly assuming some sufficient homogeneity of varieties produced within some country, for practicality of exposition. Maximizing utility subject to the budget constraint yields worker's demand in country $i$ for variety produced in country $j$

\begin{equation}
    q_{ji}=\frac{p_{ji}^{-\sigma}}{P_i^{1-\sigma}}w_i
\end{equation}

Where $P_i$ is the price index in country $i$, 
\begin{equation}
P_i=\left(\sum_k n_k p_{ki}^{1-\sigma}\right)^{\frac{1}{1-\sigma}}
\end{equation}
 where $n_k$ is the mass of firms in country $k$.

Assuming iceberg trade costs and following Mossay and Tabuchi, let matrix $\Phi$ capture the freeness of trade between all countries.  A matrix $\Phi = (\phi_{ij})_{1 \leq i,j\leq n}$ is a {\bf freeness-of-trade matrix} whenever it satisfies the following properties:
\begin{enumerate}
    \item $0 < \phi_{ij}\leq 1,$
    \item $\phi_{ij}\phi_{jk}\leq \phi_{ik},$
    \item $\phi_{ij}=\phi_{ji},$
    \item $\phi_{ii}=1.$
\end{enumerate}

The freeness-of-trade function measures the amount of loss incurred when sending something from country
$i$ to country $j.$ That is, if we send $1$ unit from $i$ to $j,$ we expect $\phi_{ij}$ units to arrive.
Many natural and artificial factors may fit into this framework,
we could expect rotting or breakage on the way to delivery if we are shipping over a long distance,
products could suffer taxes, tariffs and administrative loss, and so on.
The associated equilibrium equation to the economic model in \cite{MT15} is given by
\beq \label{MTEQ} F_i(v) = v_i - \sum L_j \phi_{ij}v_j^{-\varepsilon}  \eeq
where $\varepsilon$ is some parameter derived from the elasticity
of substitution, (specifically $\frac{\sigma}{\sigma-1}$, where $\sigma>1$ is the elasticity of substitution.) and $v_j$ is a parameter derived from wages, denoted by $\omega$, $v_i=\omega_i^\sigma$.\footnote{Utility in this model is simply wages divided by a price index, $U_i=\frac{\omega_i}{P_i}$.}
Mossay and Tabuchi showed that for nondegenerate freeness-of-trade matrices there is a unique equilibrium, that is a solution to $F = 0,$ for \eqref{MTEQ} whenever $n \leq 3.$
Their proof used the fact that $\Phi$ must be positive definite in that case.

Mossay asked whether or not $\Phi$ must be positive definite for general $n$ for nondegenerate freeness-of-trade matrices. 
We answer this question in the negative and give an example where \eqref{MTEQ} has multiple equilibria. In doing so, we characterize a family of examples involving multiple intermediaries, with applications to economic sanctions and geographical blockages. Before turning to our main results in Section 4, we introduce some mathematical preliminaries which we will draw upon in our analysis of the model described above.

\section{Mathematical Notions}

\subsection{Graph theory}
Consider a finite set of objects.
This can be individuals, towns, or countries.
We will call whatever they are \emph{vertices}.
We want to express connections between them.
This may be familiarity between individuals or adjacencies between countries.
We will call whatever these are \emph{edges}.

A \emph{graph} $G$ consists of two sets: a vertex set $V$ and an edge set $E$.
The vertex set $V$  consists of a collection of objects that we want to encode the connections between. 
 While this need not be restricted to finite sets, our assumption is that there are only finitely many actors in an economic system and so we will assume that $|V|<\infty $.  
For convenience, these will typically be represented as positive integers  e.g. if there are three vertices in a graph they will be referred to as 1, 2, and 3. 
The edge set $E$ consists of pairs $(v_1, v_2)$ where $v_1$ and $v_2$ both belong to the vertex set.
We will say a vertex $v_1$ is \emph{adjacent to} $v_2$ if $(v_1,v_2)$ is an edge of $G$ (an element of $E$).
For an introduction to the use of graph theory in economics, see ~\cite{koenig2009graph}.

In this framework, there are a few common classes of graphs. The \emph{complete graphs} $K_n$ have the vertex set $V = \{1,2,\hdots,n\}$ and the edge set $E = \{(i,j) \mid i \leq j\}$. That is, $E$ is the largest possible edge set to be placed on $V$.
 As an example, a complete graph could represent a collection of countries that can freely trade with each other. 
 Another class are the \emph{bipartite graphs} $K_{n,m}$.
 For these graphs, it is sensible to split the vertex set into two disjoint sets $V = V_1 \sqcup V_2$, where $V_1 = \{1, 2, \hdots, n\}$ and $V_2 = \{n+1, n+2, \hdots, n+m\}$.
The edge set is $E = \{(i,j) \mid i \in V_1 \text{ and } j \in V_2\}$. 
That is, each vertex in $V_1$ is adjacent to each vertex of $V_2$. However, no two vertices of $V_1$ are adjacent; this goes similarly with the vertices of $V_2$.
See Figure \ref{fig:graphs} for an example of a complete graph, $K_5$, and a bipartite graph, $K_{3,2}$. 
These bipartite graphs can be thought of as a group of countries forming two distinct \emph{anti-blocs}. 
The countries of an anti-bloc do not trade with each other due to geographical, societal, and/or political factors.
Instead, they trade with countries of the opposite anti-bloc, who also do not trade amongst themselves.

\begin{figure}
    \begin{subfigure}{0.33\linewidth}
        \begin{tikzpicture}[baseline = (n2.center), scale=0.8]
            \node (n1) at (0,0) {}; 
            \node (n2) at (-1.2,-0.8) {};
            \node (n3) at (-0.8,-2) {};
            
            \node (n4) at (1.2, -0.8) {};
            \node (n5) at (0.8,-2) {};
            
            \node[above] at (n1) {1};
            \node[left] at (n2) {2};
            \node[left] at (n3) {3};
            \node[right] at (n4) {4};
            \node[right] at (n5) {5};
            
            \filldraw[black] (n1) circle(2 pt);
            \filldraw[black] (n2) circle(2 pt);
            \filldraw[black] (n3) circle(2 pt);
            \filldraw[black] (n4) circle(2 pt);
            \filldraw[black] (n5) circle(2 pt);

            \draw[black,thick] (n1)--(n2);
            \draw[black,thick] (n1)--(n3);
            \draw[black,thick] (n1)--(n4);
            \draw[black,thick] (n1)--(n5);

            \draw[black,thick] (n2)--(n3);
            \draw[black,thick] (n2)--(n4);
            \draw[black,thick] (n2)--(n5);
            
            \draw[black,thick] (n3)--(n4);
            \draw[black,thick] (n3)--(n5);

            \draw[black,thick] (n4)--(n5);
        \end{tikzpicture}
    \end{subfigure}
    \begin{subfigure}{0.33\linewidth}
        \begin{tikzpicture}[baseline = (n2.center), scale=0.8]
            \node (n1) at (-1,0) {}; 
            \node (n2) at (-1,-1) {};
            \node (n3) at (-1,-2) {};
            
            \node (n4) at (1, -0.5) {};
            \node (n5) at (1,-1.5) {};
            
            \node[left] at (n1) {1};
            \node[left] at (n2) {2};
            \node[left] at (n3) {3};
            \node[right] at (n4) {4};
            \node[right] at (n5) {5};
            
            \filldraw[black] (n1) circle(2 pt);
            \filldraw[black] (n2) circle(2 pt);
            \filldraw[black] (n3) circle(2 pt);
            \filldraw[black] (n4) circle(2 pt);
            \filldraw[black] (n5) circle(2 pt);
            
            \draw[black,thick] (n1)--(n4);
            \draw[black,thick] (n1)--(n5);
            
            \draw[black,thick] (n2)--(n4);
            \draw[black,thick] (n2)--(n5);
            
            \draw[black,thick] (n3)--(n4);
            \draw[black,thick] (n3)--(n5);
        \end{tikzpicture}
    \end{subfigure}

    \caption{The complete graph $K_5$ (left) and the bipartite graph $K_{3,2}$. Every vertex also has a self-edge.}
    \label{fig:graphs}
\end{figure}
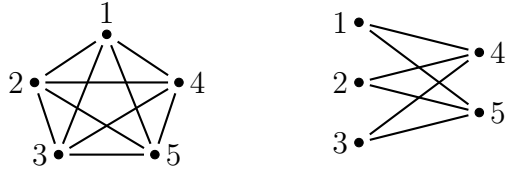

\subsection{Interpreting graphs}
Graphs can be used the encode a group of objects and connections between them.
For instance, we can use graphs to understand the clusters that individuals form through similarities or  which countries are adjacent.
 Even in a graph representation as above, self-edges need not be shown as it is known that every vertex has a self-edge. 
We can even skip any form of drawing and simply represent the graph as an \emph{adjacency matrix} $M$ of size $|V|\times |V|$. 
The $(i,j)$-th entry of $M$ is 1 if $(i,j) \in E$ and 0 otherwise.  

We will say that a graph is \emph{connected} if given  any two vertices $v$ and $w$, there exists a path $v = v_1, v_2, \hdots, v_k = w$ where $v_{i}$ is adjacent to $v_{i+1}$. 
For our purpose of discussing trading networks between countries, all graphs we consider will be connected  for the simple reason that any country that does not have trade relations does not contribute anything to the model.
In fact, it is not a large stretch to say that most countries are adjacent to each other in this sense. 
Any country with a port city could trade to any other country with a port city via sea routes.
Aside from refueling purposes, these routes do not need to go through an intermediary country between the source and destination.
As a further stretch, adding air travel suggests that all countries are adjacent to each other.   Further, even if two countries do not trade together directly, there is almost surely some \textit{indirect} trade through intermediaries, even if it is at several steps removed. 
It seems our notion of adjacency may need to be modified.

Though all countries might be considered adjacent, in the sense that they interact, some countries are more adjacent than others.
Two countries sharing a relatively large land border is safely considered ``adjacent".
However, two landlocked countries half a world away are only ``adjacent" in the loosest sense of the word.
To resolve this, we could assign a value or weight to each edge of the graph.    Instead of only encoding raw adjacency, we could have assigned a value between 0 and 1.
A value of 1 would represent ``full adjacency" and 0 would would represent ``no adjacency" (in graph theory these are referred to as weighted graphs and the matrix encoding these weights is called a weighted adjacency matrix).
Economically speaking, one can use such a weight to encode the percentage of value preserved during the transit of goods between countries, something akin to iceberg trading costs.
This labeling is immediately more relevant to the Mossay-Tabuchi model as it corresponds to the freeness-of-trade matrix.
However, it is not the only mathematically useful way  to weight the edges in our graph.  Instead of noting how efficient transit is, we could instead encode an ``economic distance," but we will first need to make the notion of distance a little more rigorous.

\vspace{.2cm}

\subsection{Metrics}
A \emph{pseudo-metric} on a set of finite points $P$ is a function $d$ that takes two points from $P$ and returns a nonnegative value with the following conditions: for all $x, y, z \in P$,
\begin{enumerate}
    \item $d(x,x) = 0$.
    \item $d(x,y) \geq 0$.
    \item $d(x,y) = d(y,x)$.
    \item $d(x,y) \leq d(x,z) + d(z,y)$.
\end{enumerate}
For completeness, a \emph{metric} is a pseudo-matrix with the added condition that $(1)$ is the only such time that the metric returns 0.   It is particularly worth noting that (pseudo-)metrics are closed under non-negative scaling and addition, which is to say that if $d_1$ and $d_2$ satisfy the conditions for being  pseudo-metrics on $P$ and $a,b$ are non-negative real numbers. then $ad_1+bd_2$ also satisfies the conditions (in order for this to work with true metrics, it is required that at least one of $a$ and $b$ is strictly positive).  Stated another way, this means that pseudo-metrics on $P$ are closed under finite, non-negative, linear combinations.  This means that the set of metrics form a cone, in the sense that if you pick any pseudo-metric and consider all of it's rescalings they will extend from the point of the cone (the all zeros pseudo-metric) to some limiting metric where everything is infinitely far apart.  To understand a cone, you need only understand its extreme rays.


Further, in the same way we might weight each edge of a graph with to represent adjacency, we can also weight each edge with the distance between two vertices.  In its simplest form, each edge value could represent distance between two countries.  The assigned weight will represent the difficulty to trade to the other country.  Since each edge represents distance, we may use a corrupted notion of an adjacency matrix.
Since this does not encode ``adjacency" we will call the corresponding matrix the \emph{distance matrix} $D$.
The $(i,j)$-th entry of $D$ represents the smallest distance between countries $i$ and $j$.
Since the distance matrix is all that matters, it is permissible to remove edges whose value holds no extra information.
That is, if there is an edge between countries $i$ and $j$ with value $D_{i,j}$, and there is a third country $k$ where $D_{ik} + D_{kj} = D_{i,j}$, then we can remove the edge $ij$ without affecting the distance matrix.

\section{Mossay-Tabuchi stability}

We say a freeness-of-trade matrix is {\bf Mossay-Tabuchi stable} whenever
$\Phi_t = (\phi_{ij}^{-\log t})_{1 \leq i,j\leq n}$
is positive semi-definite for all $0<t<1.$
We define the {\bf Mossay-Tabuchi index} to be the largest $t_0$ such that $\Phi_t$
is positive semi-definite for all $0 \leq t \leq t_0.$ Mossay-Tabuchi stability is equivalent to having a Mossay-Tabuchi
stability index of $1.$
The {\bf friction-of-trade metric} to a freeness-of-trade matrix $\Phi$ is given by the formula
    $$M_\Phi = (- \log \phi_{ij})_{1 \leq i,j\leq n} = (m_{ij})_{1 \leq i,j\leq n}.$$
Note that
\begin{enumerate}
    \item $m_{ij}\geq 0,$
    \item $ m_{ik} \leq m_{ij} + m_{jk},$
    \item $m_{ij}=m_{ji},$
    \item $m_{ii}=0,$
\end{enumerate}
and thus the entries of $M_{\Phi}$ define a pseudo-metric. Note that the set of pseudo-metrics form a convex cone. In the case where $\Phi$ is nondegenerate,
$M_{\Phi}$ is indeed a bona fide metric on points $\{1,\ldots, n\}.$

{Note that if the associated pseudo-metric is indeed a metric, the freeness-of-trade matrix $\Phi_t$ will have entries off the main diagonal that all vary with $t$ and the asymptotic behavior of the eigenvalues can be easily understood.  As $t$ tends to 0, the freeness-of-trade matrix  tends to the identity matrix and all of the eigenvalues tend to 1.  As $t$ tends to 1, the matrix tends to the matrix whose entries are all 1 and the eigenvalues are $n$ with multiplicity 1 and $0$ with multiplicity $n-1$.  Since the eigenvalues will asymptotically all be 1 as $t\to0$ each freeness-of-trade matrix must eventually have all positive eigenvalues.  Thus the Mossay-Tabuchi index is well-defined.  Furthermore, if a ray in the cone of metrics will have an associated freeness-of-trade matrix that is not positive semi-definite, it will be near $t=1$ and thus closer to perfect trading efficiency between trading partners.}

A matrix is {\bf conditionally negative semi-definite} if
    $$\sum m_{ij} c_ic_j \leq 0$$
whenever $\sum c_i = 0.$
The following observation is restatement of Schoenberg's theorem from \cite[pp. 129--136]{Paulsen}. 
\begin{theorem}
    Let $\Phi$ be a freeness-of-trade matrix.
    The following are equivalent:
    \begin{enumerate}
        \item $\Phi$ is Mossay-Tabuchi stable,
        \item $M_{\Phi}$ is conditionally negative semi-definite,
        \item $M_\Phi = (\|x_i-x_j\|_2^2)_{1 \leq i,j\leq n}$ for some vectors $x_i \in \mathbb{R}^n.$
    \end{enumerate}
\end{theorem}
The final condition can be interpreted as a metric proportional to the energy required to transport matter from $x_i$ to $x_j$ in straight line in a fixed amount of time. Such metrics are often called {\bf metrics of negative type} \citep{naor,naor2,acmnotes, khot2015unique}. Note that any metric space which isometrically embeds into $L^1$ is of negative type \citep{naor}. Moreover, any metric space on $4$ points embeds in $L^1.$ One expects some results on metrics of negative type can be generalized or relaxed to those with some fixed Mossay-Tabuchi index. To look for degenerate examples, it is natural to look at the extreme rays of the cone of metrics \citep{grish, avis}.

\begin{figure}[h]

   \begin{subfigure}{0.24\linewidth}

    \begin{tikzpicture}[baseline = (n2.center), scale=0.8]
        \node (n1) at (-1,0) {}; 
        \node (n2) at (-1,-1) {};
        \node (n3) at (-1,-2) {};
        
        \node (n4) at (1, -0.5) {};
        \node (n5) at (1,-1.5) {};
        
        \node[left] at (n1) {1};
        \node[left] at (n2) {2};
        \node[left] at (n3) {3};
        \node[right] at (n4) {4};
        \node[right] at (n5) {5};
        
        \filldraw[black] (n1) circle(2 pt);
        \filldraw[black] (n2) circle(2 pt);
        \filldraw[black] (n3) circle(2 pt);
        \filldraw[black] (n4) circle(2 pt);
        \filldraw[black] (n5) circle(2 pt);
        
        \draw[black,thick] (n1)--(n4);
        \draw[black,thick] (n1)--(n5);
        
        \draw[black,thick] (n2)--(n4);
        \draw[black,thick] (n2)--(n5);
        
        \draw[black,thick] (n3)--(n4);
        \draw[black,thick] (n3)--(n5);
    \end{tikzpicture}
       \end{subfigure}  
    \quad
    \begin{subfigure}{0.28\linewidth}
    $\begin{bmatrix}
    0&2&2&1&1\\
    2&0&2&1&1\\
    2&2&0&1&1\\
    1&1&1&0&2\\
    1&1&1&2&0
    \end{bmatrix}$
    \end{subfigure} 
    \quad 
    \begin{subfigure}{0.28\linewidth}
    $\begin{bmatrix}
    1&t^{2}&t^{2}&t&t\\
    t^{2}&1&t^{2}&t&t\\
    t^{2}&t^{2}&1&t&t\\
    t&t&t&1&t^{2}\\
    t&t&t&t^{2}&1
\end{bmatrix}$
\end{subfigure}

    \caption{The $K_{3,2}$ graph (left) along with its corresponding metric matrix $M_{\Phi^{K_{3,2}}}$ (center) and freeness of trade matrix $\Phi^{K_{3,2}}_{t}$ (right). The entries of the metric are given by the length of the shortest path between two vertices.}
    \label{f:k32}
\end{figure}
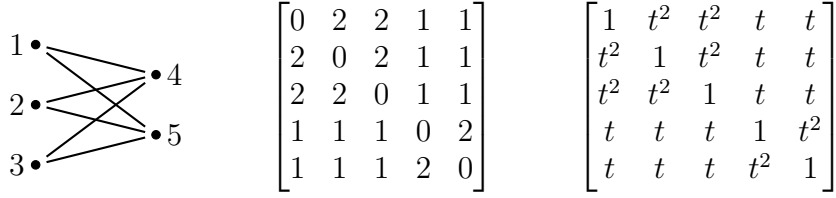

The $K_{n,m}$ bipartite metric on $n+m$ points is given by $m_{i,j} = 1$ if $i\leq n, j>n$
and $m_{i,j}=2$ otherwise if $i\neq j.$ We denote the corresponding freeness-of-trade matrix
by  $\Phi^{K_{n,m}}.$ Figure \ref{f:k32} shows the correspondence between the trade network graph, metric matrix, and freeness-of-trade matrix for the 3 by 2 bipartite case. 
\vspace{1cm}
\begin{theorem}\label{bipspec}
The stability index of $\Phi^{K_{n,m}}$ is equal to $(\sqrt{(m-1)(n-1)})^{-1}.$

If $n, m \geq 2, n+m \geq 5$ then $\Phi^{K_{n,m}}$ is not Mossay-Tabuchi stable.
\end{theorem}
\begin{proof}
	The matrix $\Phi^{K_{n,m}}_{t} - (1-t^2)I$ has rank 2, implying that $1-t^2$ is an eigenvalue of $\Phi^{K_{n,m}}_t$ with multiplicity $n+m-2$.
    The remaining two eigenvalues, $\lambda_1$ and $\lambda_2$, can be found by examining the trace of $\Phi^{K_{n,m}}_{t}$ and $(\Phi^{K_{n,m}}_{t})^2$:
    \begin{align*}
        & \Tr(\Phi^{K_{n,m}}_{t}) = \lambda_1 + \lambda_2 + (n+m-2)(1-t^2) = n+m.\\
        & \Tr((\Phi^{K_{n,m}}_{t})^2) = \lambda_1^2 + \lambda_2^2 + (n+m-2)(1-t^2)^2 \hspace{4.5cm}\\ 
        & \hspace{2.35cm} = n+m + (n-1)^2t^4 + (m-1)^2t^4 + 2nm t^2.
    \end{align*}
    Solving the system shows the location of the remaining eigenvalues:
    $$1+\frac{n+m-2}{2}t^2 \pm \frac{\sqrt{4nmt^2+(n-m)^2t^4}}{2}.$$
    Note that $t^{-1}=\sqrt{(m-1)(n-1)}$ solves the equation for $0.$
\end{proof}
Thus, bipartiteness manifests as a negative eigenvalue for $t$ near $1.$ One expects there is a general spectral theory of metrics and freeness-of-trade matrices. There are some results along these lines, for example \cite{steinerberger}.

We now show that for certain bipartite metrics, among many others with similar inconvenience structure, perverse equilibria exist.
\begin{theorem}
Let $\Phi$ be a freeness-of-trade matrix.
Let $\vec{a}$ be the all ones vector.
Let $\vec{b}$ be a $\pm 1$ vector.
Let $L$ be a diagonal matrix such that $\Phi L$ has $\vec{a}$ and $\vec{b}$ as eigenvectors.
Let $\lambda_a$ and $\lambda_b$ be the corresponding eigenvalues.
(For example, $\Phi^{K_{n,n}}_t$ for small enough $t.$)
If $\varepsilon \lambda_b/\lambda_a < -1,$ then the model \eqref{MTEQ} does not admit a unique solution.
\end{theorem}
\begin{proof}
Substitute $w=v^{-\varepsilon}.$
With these parameters we want to solve $w_i^{-1/\varepsilon} = \sum \phi_{ij}L_jw_j.$
Let $\vec{a}$ be the all ones eigenvector of $\Phi$ and let $\vec{b}$ be the eigenvector with ones and negative ones. Let $\lambda_a$ and $\lambda_b$ be the corresponding eigenvalues. Set $w = x\vec{a}+y\vec{b}.$
Substituting in we get the system of equations
	$$1/(x+y)=(\lambda_ax+\lambda_b y)^\varepsilon,$$
	$$1/(x-y)=(\lambda_ax-\lambda_b y)^\varepsilon.$$
Set $z=y/x$ and $\tau = \lambda_b/\lambda_a.$ Note, by the Perron-Frobenius theorem, we must have $|\tau|< 1.$ Dividing the two equations we get
$$(1-z)/(1+z)=(1+\tau z)^\varepsilon/(1-\tau z)^\varepsilon.$$
Make the substitution $u=(1+\tau z)/(1-\tau z).$
We get
    $$(1 + \tau  - u + \tau u)/(-1 + \tau + u + \tau u)=u^\varepsilon.$$
So,
$$0=u^\varepsilon(-1 + \tau + u + \tau u)-(1 + \tau  - u + \tau u).$$
The case $u=1$ trivially gives a solution. Dividing by $u-1$
gives
    $$0=u^\varepsilon + 1 + \tau (u+1)\frac{u^\varepsilon-1}{u-1}.$$
Note, as $u$ goes to $1$ the right hand side is equal to 
    $$2+2\tau \varepsilon.$$
Moreover, the limit as we go to infinity is $+\infty.$
Thus, by the intermediate value theorem, if $\tau\varepsilon<-1,$
then there must be a nontrivial solution for $u.$
Since $u$ is positive, and $u^\varepsilon = (1-z)/(1+z),$ we see $z$ must be between $-1$ and $1.$
(If $\tau$ is near $-1$ this translates to $w$ being approximately some rescaling of $\vec{a}\pm\vec{b},$ which in the model roughly corresponds to all the wages/utility/resources being in one of the bipartite components. Note that as $\varepsilon$ becomes larger, which corresponds to the elasticity of substitution $\sigma$ being near $1,$ it is easier to find perverse equilibria.)  
\end{proof}

\subsection{An example}
Let 
$$\Phi = \bbm
1 & t^2 & t^2 & t  & t & t \\
t^2 & 1 & t^2 & t & t & t \\
t^2 & t^2 & 1 & t & t & t \\
t& t & t & 1 & t^2 & t^2 \\
t & t & t & t^2 & 1 & t^2\\
t & t & t & t^2 & t^2 & 1
\ebm.$$
Set all $L_i$ to be $1.$
The all ones vector has eigenvalue $(1+2t)(1+t),$ the vector with $1$'s in the first three slots and $-1$'s in the latter three has eigenvalue $(1-2t)(1-t).$
Solving as in the proof above, we get
     $$0=u^\varepsilon + 1 + \frac{(1-2t)(1-t)}{(1+2t)(1+t)}(u+1)\frac{u^\varepsilon-1}{u-1}.$$
To guarantee perverse equilibria we need
    $$\frac{(1-2t)(1-t)}{(1+2t)(1+t)}\varepsilon < -1$$
to be negative. That is, $\varepsilon > \abs{\frac{(1+2t)(1+t)}{(1-2t)(1-t)}}.$
\vspace{0.1cm}

To see what these equilibria look like asymptotically, consider the case when $\varepsilon$ goes to $\infty.$ 
Rewriting our previous solution 
$$-1= \frac{(1-2t)(1-t)}{(1+2t)(1+t)}\frac{(u+1)(u^\varepsilon-1)}{(u-1)(u^{\varepsilon}+1)}$$
which tends to 
$$\frac{1-u}{1+u} = \pm \tau$$
which is solved by
$u= \frac{1-\tau}{1+\tau}, \frac{1+\tau}{1-\tau},$
which imply $z=\pm 1,$ which implies $x= \pm y,$ and thus the solution for the $w$ will be approximately a $0-1$ vector, which corresponds to runaway inequality in terms of utility, wages, and so on.

\section{Economic Sanctions}

In this section, we provide illustrative examples of the analysis above to the study of economic sanctions. Because sanctions drastically reduce  trade between targets and senders, they may create bipartite trade networks such as those discussed above, resulting in multiple perverse equilibria. A key factor in these examples is that there are (at least) two sets of countries which do not trade heavily with one another, and instead trade, if it occurs, must go through an intermediary. This can arise the context of economic sanctions, where the sender and target do not trade directly, but instead trade is conducted through third party intermediaries. This is particularly the case when the sanctions-busting state do not have strong trade ties with one another.

We are unable to specify exactly which actors will benefit and which will suffer in a distributional sense due to multiple equilibria, but we can characterize those multiple equilibria, such as the three identified in the previous example as ``perverse'' in that they each exhibit rent seeking, distributional conflict at the cost of general welfare. Indeed, it is exactly in the presence of multiple equilibria that we may expect policy to have the most relevance \citep{Benhabib2001}, an issue we turn to in the next section.

 Below, we discuss such cases, and contrast them with cases which more closely resemble the unique equilibrium under Mossay-Tabuchi stability. In doing so, we provide examples of how the structure of international trade, in particular the ``geometry of inconvenience'' conditions the impact of economic sanctions.

 In these examples, we explore substantive applications of the bipartite family to illustrate the intuition and applicability of the above results. This is but one of many examples of the result presented above, characterised by “anti-blocs.” That is, two “blocs” of countries, which trade exclusively with members of the other bloc, and not with members of their own. Such arrangements have arisen historically due to political factors such as mercantilism and economic sanctions, but also naturally due to other factors such as geography and climate. 
 
 Below we draw on the previous results to guide future research on how the geometry of inconvenience effects the presence in network based rent-seeking.  In other words, how the particular network structure of international trade influences economic and political outcomes. For a seminal treatment of the role of state power in determining the structure of international trade, see \cite{krasner_1976}. We provide a novel, trade-network structure based, general explanation for such rent-seeking behavior to influence the structure of international trade. This result may provide a fruitful path forward for reconciling general equilibrium results in trade theory with results in noncooperative game-theory regarding political rent seeking beyond triangular dead-weight loss.\footnote{Rent-seeking refers to expropriating activities that bring positive returns to the individual but which are detrimental to general welfare \citep{krueger_1974}.} The multiple equilibria we characterize provide a context which provide the space for such competition to occur.

Sanctions may establish two such “anti-blocs” as those identified above. The first comprises the sender(s) and target(s) of economic sanctions. By design, sanctions impact the ease of trade negatively between these states. 
The other bloc consists of sanctions busters, those states which serve as an intermediary to circumvent sanctions. While the implementation of sanctions is politically motivated to coerce or signal intentions in crisis bargaining, sanctions busters are often commercially motivated \citep{early_2015} and firms respond to the market incentives to invest in potential sanctions busting states \citep{barry_kleinberg_2015}. The distortions created by sanctions create opportunities for sanctions busters to exploit the target of sanctions as well firms in the sender state. Such opportunities for rent seeking drive the perverse equilibria we identify above. Our contribution is to show that not all sanctions episodes are created equally in this regard. Furthermore, the presence of multiple equilibria present the opportunity for bargaining among policymakers \citep{venables_1984}. Third parties have strong incentives to circumvent economic sanctions, maintaining and expanding trade relationships with both target and sanctioning states. In doing so, they may be able to extract rents from both.  

Examples include US sanctions on Iran, with the United Arab Emirates acting as a sanctions-buster, US sanctions on South Africa, with the UK and others acting as sanctions busters. Here, the intermediaries voiced support of the sanctions, while still allowing trade to continue for their own domestic firms \citep{early_2009}.

One condition that was identified above was the presence of multiple intermediaries that do not have low-cost trade with one another. The basic insight here is that exchange between the intermediaries may temper their ability to seek rents. In other words, rent seeking should be most severe when there are multiple sanctions busters that are relatively isolated from one another. However, the condition is a matter of degree, and admits some trade between sanctions busters.

This suggests that researchers should move beyond the conventional approach of simply counting the number of sanctions busters, and consider the trade networks between them. Interestingly, fostering trade between sanctions busting states may help curtail rent seeking by them, therefore enhancing the overall efficacy of the sanctions regime. 

Similarly, our results have implications for debates regarding the relative effectiveness of bilateral versus multilateral sanctions \citep{weber2020, miers2002, martin1994, bapat2009}. For instance, \cite{Kaempfer1999} presents evidence that multilateral sanctions often fail due to rent seeking in the target country. We provide important theoretical nuance to this claim.  In the case of bilateral economic sanctions, with one or two sanctions busters perverse equilibria do not exist. However with three or more sanctions busting states they may, depending on the ease of trade between them. In the case of multilateral economic sanctions,  if there is only one sanctions busting state no perverse equilibria exist. However with more than one, perverse equilibria may exist. This highlights an interaction between the number of sanctioning states and the number of sanctions busters that has not been identified previously. 

Our geometry of inconvenience also yields new implications for the number of targeted states. This is important because in the current context of extensive use of economic sanctions as a tool of foreign influence \citep{aidt2022}, many states are targeted simultaneously \citep{drezner2021united}. 
Our results suggest that the more targets of sanctions there are simultaneously, the more likely perverse equilibria are to exist. 

Such considerations are relevant to recent work quantifying the misery and other consequences imposed by economic sanctions \citep{early2022, Ozdamar2021, morgan2023, allen2020, kavalki2020}, particularly because the rent seeking established due to such inefficiencies may persist after sanctions are lifted \citep{Andreas2005, pond2017}. Vested interests form when perverse equilibria exist in this context, and once formed they can be difficult to displace \citep{barry_kleinberg_2015, early2019searching, dorussen2001}. We next turn to an analysis of policy options available to states, to consider the conditions under which they may be able to escape such perverse equilibria.

      %
        
        
        
        
        

\section{Policy Analysis} 
In this section, we use a series of numerical simulations to illustrate how changes to the trade structure, via policy change or other shocks, influence whether the freeness of trade matrix is Mossay-Tabuchi  stable. To do so, we will consider whether a specific trade network configuration results in a unique equilibrium of the type identified by Mossay and Tabuchi, or whether there are perverse, multiple equilibria as we identify above. For instance, the bipartite $K_{4,2}$ trade network results in perverse equilibria even when trade costs are high, such cases are labeled in red below. The $K_{4,2}$ trade network is the top vertex of Figure 3. Similarly, the bipartite $K_{3,3}$ is also not Mossay-Tabuchi stable, even under very high trade costs, and is the bottom right vertex of Figure \ref{fB}. The red color at this vertex indicates that unique equilibrium does not exist in this network. $C_{5,1}$ represents a network in which one member of the right hand side anti-bloc has left. This forms the bottom left vertex of Figure \ref{fB}.    

For each of the images presented in Figures \ref{fB}-\ref{fG}, each vertex of a triangle represents one of the extreme rays of the metric cone, out another way, it represents a specific trade network. Every point of the triangle represents some weighted average of the extreme rays (e.g. an interior point might be $.25*M_1+.25*M_2+.5*M_3$). This allows us to consider the impact of moving from one trade network to another.   A point in the triangle is blue if the specific weighted average is it is Mossay-Tabuchi stable and red if it is not. What this illustrates is the amount of influence an individual actor in a network has as the asymmetry in the anti-blocs is altered.  Each figure is labeled so that the bottom left metric is listed first, then the top metric, and ending with the bottom right. 

In the figures that follow, blue coloring means the corresponding metric matrix will produce a non-perverse freeness-of-trade matrix regardless of what nonzero scaling you give the metric matrix. The red areas are not Mossay-Tabuchi stable. 

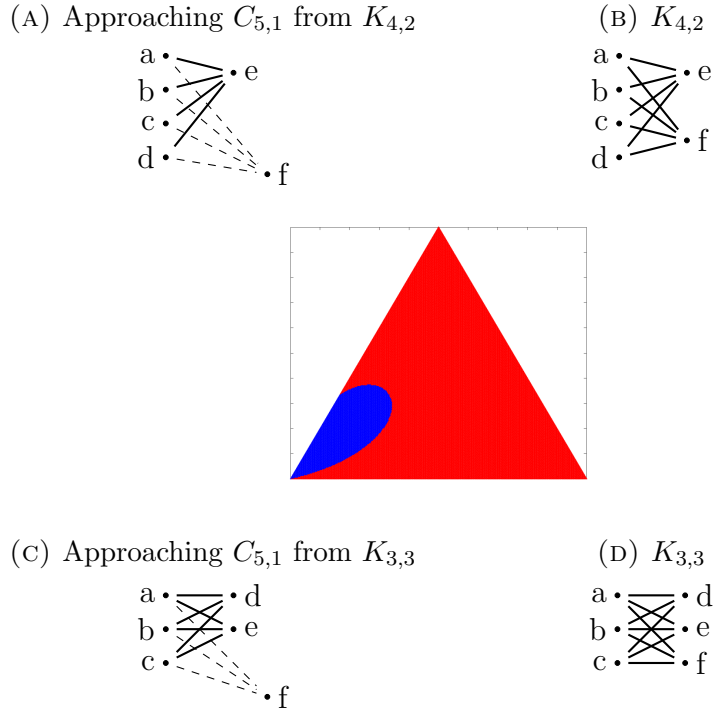
\begin{figure}[H]
   \begin{subfigure}[t]{0.45\linewidth}
       \caption{Approaching $C_{5,1}$ from $K_{4,2}$}
    \centering
       \begin{tikzpicture} [baseline = (n2.center), scale=0.45]
        \node (n1) at (-1,0) {}; 
        \node (n2) at (-1,-1) {};
        \node (n3) at (-1,-2) {};
        \node (n4) at (-1, -3) {};
        
        \node (n5) at (1,-0.5) {};
        \node (n6) at (2,-3.5) {};
        
        \node[left] at (n1) {a};
        \node[left] at (n2) {b};
        \node[left] at (n3) {c}; 
        \node[left] at (n4) {d};
        \node[right] at (n5) {e};
         \node[right] at (n6) {f};
        
        \filldraw[black] (n1) circle(2 pt);
        \filldraw[black] (n2) circle(2 pt);
        \filldraw[black] (n3) circle(2 pt);
        \filldraw[black] (n4) circle(2 pt);
        \filldraw[black] (n5) circle(2 pt);
        \filldraw[black] (n6) circle(2 pt);
        
        \draw[black,thick] (n1)--(n5);
        \draw[black,dashed] (n1)--(n6);

        \draw[black,thick] (n2)--(n5);
        \draw[black,dashed] (n2)--(n6);

        \draw[black,thick] (n3)--(n5);
        \draw[black,dashed] (n3)--(n6);

        \draw[black,thick] (n4)--(n5);
        \draw[black,dashed] (n4)--(n6);

    \end{tikzpicture}
   \end{subfigure}
   \begin{subfigure}[t]{0.45\linewidth}
    \centering
    \caption{$K_{4,2}$}
    \centering
    \begin{tikzpicture} [baseline = (n2.center), scale=0.45]
        \node (n1) at (-1,0) {}; 
        \node (n2) at (-1,-1) {};
        \node (n3) at (-1,-2) {};
        \node (n4) at (-1, -3) {};
        
        \node (n5) at (1,-0.5) {};
        \node (n6) at (1,-2.5) {};
        
        \node[left] at (n1) {a};
        \node[left] at (n2) {b};
        \node[left] at (n3) {c}; 
        \node[left] at (n4) {d};
        \node[right] at (n5) {e};
         \node[right] at (n6) {f};
        
        \filldraw[black] (n1) circle(2 pt);
        \filldraw[black] (n2) circle(2 pt);
        \filldraw[black] (n3) circle(2 pt);
        \filldraw[black] (n4) circle(2 pt);
        \filldraw[black] (n5) circle(2 pt);
        \filldraw[black] (n6) circle(2 pt);
        
        \draw[black,thick] (n1)--(n5);
        \draw[black,thick] (n1)--(n6);

        \draw[black,thick] (n2)--(n5);
        \draw[black,thick] (n2)--(n6);

        \draw[black,thick] (n3)--(n5);
        \draw[black,thick] (n3)--(n6);

        \draw[black,thick] (n4)--(n5);
        \draw[black,thick] (n4)--(n6);

    \end{tikzpicture}

   \end{subfigure}

  \begin{center}
   \begin{subfigure}[c]{0.40\linewidth}
       \includegraphics[width=\linewidth]{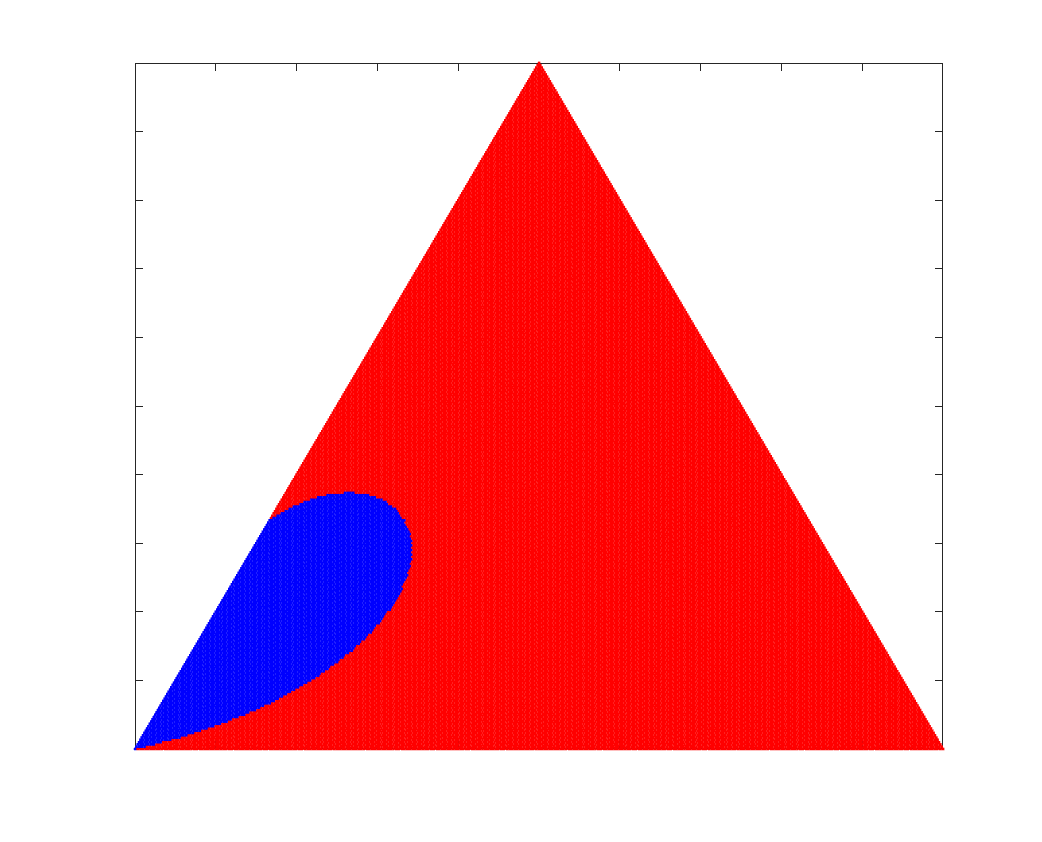}

   \end{subfigure}
    \end{center}
   
    \begin{subfigure}[t]{0.45\linewidth}
       \caption{Approaching $C_{5,1}$ from $K_{3,3}$}
    \centering
    \begin{tikzpicture} [baseline = (n2.center), scale=0.45]
        \node (n1) at (-1,0) {}; 
        \node (n2) at (-1,-1) {};
        \node (n3) at (-1,-2) {};
        
        \node (n4) at (1, 0) {};
        \node (n5) at (1,-1) {};
        \node (n6) at (2,-3) {};
        
        \node[left] at (n1) {a};
        \node[left] at (n2) {b};
        \node[left] at (n3) {c}; 
        \node[right] at (n4) {d};
        \node[right] at (n5) {e};
         \node[right] at (n6) {f};
        
        \filldraw[black] (n1) circle(2 pt);
        \filldraw[black] (n2) circle(2 pt);
        \filldraw[black] (n3) circle(2 pt);
        \filldraw[black] (n4) circle(2 pt);
        \filldraw[black] (n5) circle(2 pt);
        \filldraw[black] (n6) circle(2 pt);
        
        \draw[black,thick] (n1)--(n4);
        \draw[black,thick] (n1)--(n5);
        \draw[black,dashed] (n1)--(n6);
        
        \draw[black,thick] (n2)--(n4);
        \draw[black,thick] (n2)--(n5);
        \draw[black,dashed] (n2)--(n6);
        
        \draw[black,thick] (n3)--(n4);
        \draw[black,thick] (n3)--(n5);
        \draw[black,dashed] (n3)--(n6);
        
    \end{tikzpicture}
   \end{subfigure}
     \begin{subfigure}[t]{0.45\linewidth}
    \centering
    \caption{$K_{3,3}$}
    \begin{tikzpicture} [baseline = (n2.center), scale=0.45]
        \node (n1) at (-1,0) {}; 
        \node (n2) at (-1,-1) {};
        \node (n3) at (-1,-2) {};
        
        \node (n4) at (1, 0) {};
        \node (n5) at (1,-1) {};
        \node (n6) at (1,-2) {};
        
        \node[left] at (n1) {a};
        \node[left] at (n2) {b};
        \node[left] at (n3) {c}; 
        \node[right] at (n4) {d};
        \node[right] at (n5) {e};
         \node[right] at (n6) {f};
        
        \filldraw[black] (n1) circle(2 pt);
        \filldraw[black] (n2) circle(2 pt);
        \filldraw[black] (n3) circle(2 pt);
        \filldraw[black] (n4) circle(2 pt);
        \filldraw[black] (n5) circle(2 pt);
        \filldraw[black] (n6) circle(2 pt);
        
        \draw[black,thick] (n1)--(n4);
        \draw[black,thick] (n1)--(n5);
        \draw[black,thick] (n1)--(n6);
        
        \draw[black,thick] (n2)--(n4);
        \draw[black,thick] (n2)--(n5);
        \draw[black,thick] (n2)--(n6);
        
        \draw[black,thick] (n3)--(n4);
        \draw[black,thick] (n3)--(n5);
        \draw[black,thick] (n3)--(n6);
        
    \end{tikzpicture}
   \end{subfigure}
   
   \caption{\footnotesize{The impact of one state leaving an antibloc on MT Stability. The vertices are $C_{5,1}$ (Bottom Left), $K_{4,2}$ (Top), $K_{3,3}$ (Bottom Right). Interior points are linear combinations of the networks at the vertices. Neither $K_{4,2}$ nor $K_{3,3}$ are MT Stable, one state leaving the trade network from the right hand side anti-bloc induces MT Stability in the $K_{4,2}$ case, but not the $K_{3,3}$ case. The diagrams A-D provide examples of the networks at the vertices.} } \label{fB}
    \end{figure}

We first turn to the question of the impact of one state unilaterally erecting barriers to trade. We refer to this below as a cut metric. In particular, we are interested in whether one state leaving the trade network can induce Mossay-Tabuchi stability. Here, we find that the structure of international trade is crucial to answering this question, and thus broad statements about multilateral vs bilateral sanctions and the number of sanctions-busting states, which are often made by the literature \citep{Kaempfer1999, bapat2009}, mask network geometry specific effects. To understand the implications of changes in policy, the structure of international trade must be considered.

In Figure \ref{fB}, we can see that the cut metric that adds a uniform distance between the first five points and the sixth is pivotal and can induce Mossay-Tabuchi stability when approaching from $K_{4,2}$.  While this alteration cannot force a K$_{3,3}$ metric to become Mossay-Tabuchi stable, it \textit{can} force the K$_{4,2}$ metric to be. 
 This is because this alteration cuts off one of the small anti-bloc on K$_{4,2}$.

\begin{figure}[H]
 \begin{subfigure}[t]{0.45\linewidth}
       \caption{Approaching $C_{1,5}$ from $K_{4,2}$}
    \centering
    \begin{tikzpicture}[baseline = (n2.center), scale=.45]
        \node (n1) at (-2,1) {}; 
        \node (n2) at (-1,-1) {};
        \node (n3) at (-1,-2) {};
        \node (n4) at (-1, -3) {};
        
        \node (n5) at (1,-0.5) {};
        \node (n6) at (1,-2.5) {};
        
        \node[left] at (n1) {a};
        \node[left] at (n2) {b};
        \node[left] at (n3) {c}; 
        \node[left] at (n4) {d};
        \node[right] at (n5) {e};
         \node[right] at (n6) {f};
        
        \filldraw[black] (n1) circle(2 pt);
        \filldraw[black] (n2) circle(2 pt);
        \filldraw[black] (n3) circle(2 pt);
        \filldraw[black] (n4) circle(2 pt);
        \filldraw[black] (n5) circle(2 pt);
        \filldraw[black] (n6) circle(2 pt);

        \draw[black,dashed] (n1)--(n5);
        \draw[black,dashed] (n1)--(n6);

        \draw[black,thick] (n2)--(n5);
        \draw[black,thick] (n2)--(n6);

        \draw[black,thick] (n3)--(n5);
        \draw[black,thick] (n3)--(n6);
        
          \draw[black,thick] (n4)--(n5);
        \draw[black,thick] (n4)--(n6);

    \end{tikzpicture}
   \end{subfigure}
   \begin{subfigure}[t]{0.45\linewidth}
    \centering
    \caption{$K_{4,2}$}
    \centering
    \begin{tikzpicture} [baseline = (n2.center), scale=.45]
        \node (n1) at (-1,0) {}; 
        \node (n2) at (-1,-1) {};
        \node (n3) at (-1,-2) {};
        \node (n4) at (-1, -3) {};
        
        \node (n5) at (1,-0.5) {};
        \node (n6) at (1,-2.5) {};
        
        \node[left] at (n1) {a};
        \node[left] at (n2) {b};
        \node[left] at (n3) {c}; 
        \node[left] at (n4) {d};
        \node[right] at (n5) {e};
         \node[right] at (n6) {f};
        
        \filldraw[black] (n1) circle(2 pt);
        \filldraw[black] (n2) circle(2 pt);
        \filldraw[black] (n3) circle(2 pt);
        \filldraw[black] (n4) circle(2 pt);
        \filldraw[black] (n5) circle(2 pt);
        \filldraw[black] (n6) circle(2 pt);
        
        \draw[black,thick] (n1)--(n5);
        \draw[black,thick] (n1)--(n6);

        \draw[black,thick] (n2)--(n5);
        \draw[black,thick] (n2)--(n6);

        \draw[black,thick] (n3)--(n5);
        \draw[black,thick] (n3)--(n6);

        \draw[black,thick] (n4)--(n5);
        \draw[black,thick] (n4)--(n6);

    \end{tikzpicture}
   \end{subfigure}

   \begin{center}
   \begin{subfigure}[]{0.40\linewidth}
       \includegraphics[width=\linewidth]{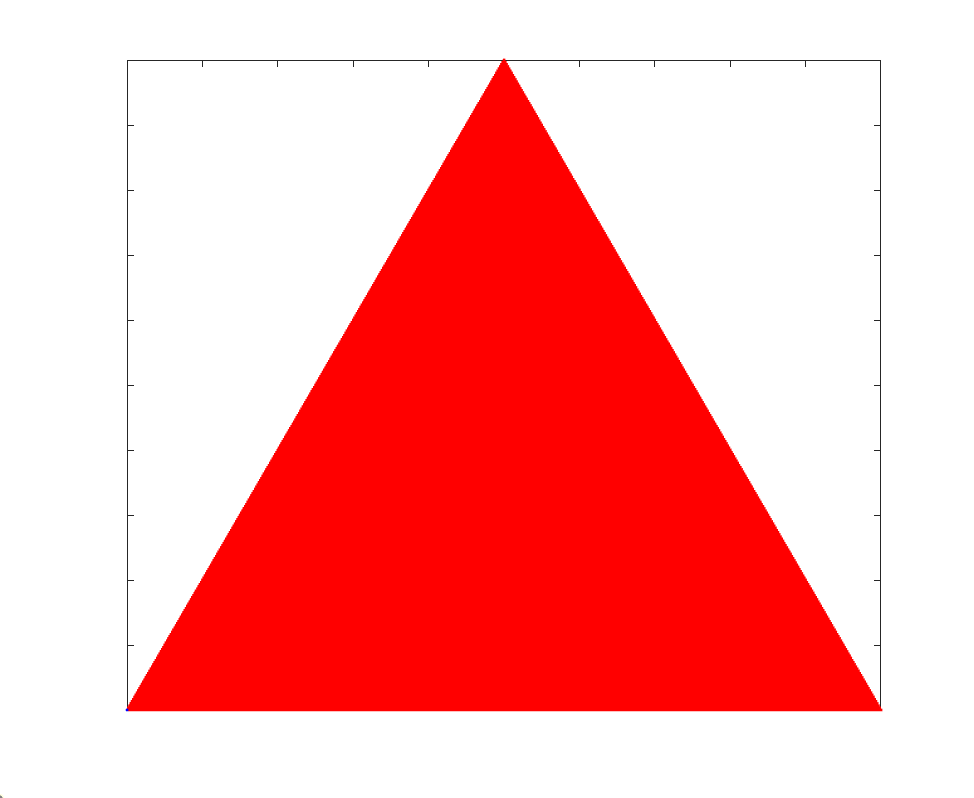}
       
   \end{subfigure}
   \end{center}
   
   \begin{subfigure}[t]{0.45\linewidth}
       \caption{Approaching $C_{1,5}$ from $K_{3,3}$}
    \centering
    \begin{tikzpicture} [ baseline = (n2.center),  scale=.45]
        \node (n1) at (-2,1) {}; 
        \node (n2) at (-1,-1) {};
        \node (n3) at (-1,-2) {};
        
        \node (n4) at (1, 0) {};
        \node (n5) at (1,-1) {};
        \node (n6) at (1,-2) {};
        
        \node[left] at (n1) {a};
        \node[left] at (n2) {b};
        \node[left] at (n3) {c}; 
        \node[right] at (n4) {d};
        \node[right] at (n5) {e};
         \node[right] at (n6) {f};
        
        \filldraw[black] (n1) circle(2 pt);
        \filldraw[black] (n2) circle(2 pt);
        \filldraw[black] (n3) circle(2 pt);
        \filldraw[black] (n4) circle(2 pt);
        \filldraw[black] (n5) circle(2 pt);
        \filldraw[black] (n6) circle(2 pt);
        
        \draw[black,dashed] (n1)--(n4);
        \draw[black,dashed] (n1)--(n5);
        \draw[black,dashed] (n1)--(n6);
        
        \draw[black,thick] (n2)--(n4);
        \draw[black,thick] (n2)--(n5);
        \draw[black,thick] (n2)--(n6);
        
        \draw[black,thick] (n3)--(n4);
        \draw[black,thick] (n3)--(n5);
        \draw[black,thick] (n3)--(n6);
        
    \end{tikzpicture}
   \end{subfigure}
     \begin{subfigure}[t]{0.45\linewidth}
    \centering
    \caption{$K_{3,3}$}
    \begin{tikzpicture}[baseline = (n2.center), scale=0.45]
        \node (n1) at (-1,0) {}; 
        \node (n2) at (-1,-1) {};
        \node (n3) at (-1,-2) {};
        
        \node (n4) at (1, 0) {};
        \node (n5) at (1,-1) {};
        \node (n6) at (1,-2) {};
        
        \node[left] at (n1) {a};
        \node[left] at (n2) {b};
        \node[left] at (n3) {c}; 
        \node[right] at (n4) {d};
        \node[right] at (n5) {e};
         \node[right] at (n6) {f};
        
        \filldraw[black] (n1) circle(2 pt);
        \filldraw[black] (n2) circle(2 pt);
        \filldraw[black] (n3) circle(2 pt);
        \filldraw[black] (n4) circle(2 pt);
        \filldraw[black] (n5) circle(2 pt);
        \filldraw[black] (n6) circle(2 pt);
        
        \draw[black,thick] (n1)--(n4);
        \draw[black,thick] (n1)--(n5);
        \draw[black,thick] (n1)--(n6);
        
        \draw[black,thick] (n2)--(n4);
        \draw[black,thick] (n2)--(n5);
        \draw[black,thick] (n2)--(n6);
        
        \draw[black,thick] (n3)--(n4);
        \draw[black,thick] (n3)--(n5);
        \draw[black,thick] (n3)--(n6);
        
    \end{tikzpicture}
   \end{subfigure}

   \caption{\footnotesize{The impact of one state leaving an antibloc on MT Stability. The vertices are $C_{1,5}$ (Bottom Left), $K_{4,2}$ (Top), $K_{3,3}$ (Bottom Right). Interior points are linear combinations of the networks at the vertices. Neither $K_{4,2}$ nor $K_{3,3}$ are MT Stable, one state leaving the trade network from the left hand side anti-bloc does not induce MT Stability in the $K_{4,2}$ case, nor the $K_{3,3}$ case. The diagrams A-D provide examples of the networks at the vertices.}}\label{fA}
   \end{figure}
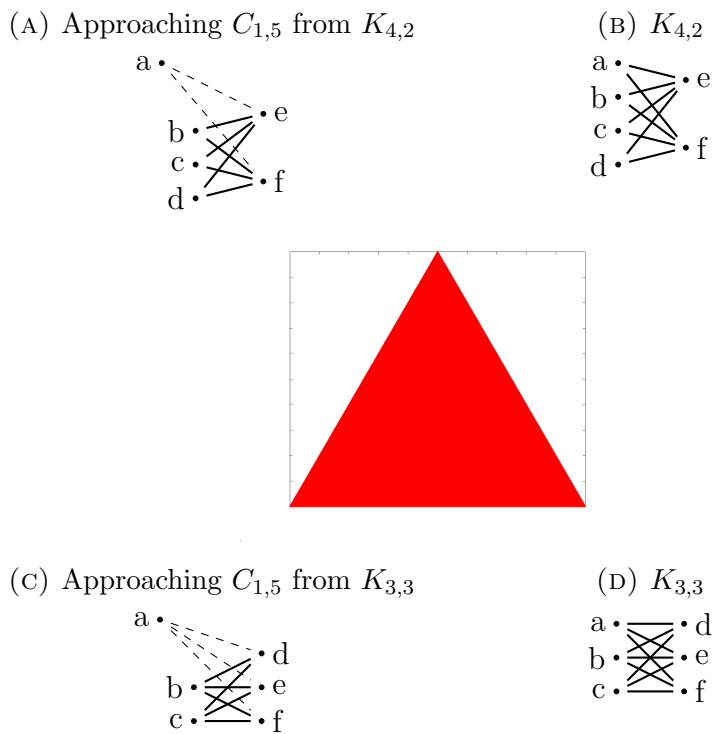

 In Figure \ref{fA} we can see that the cut metric that adds a uniform distance between the first point and the remaining five cannot alter the fundamental problematic nature of weighted averages of $K_{4,2}$ and $K_{3,3}$ as in essence this only cuts off one member of the larger anti-bloc and leaves some average of $K_{3,2}$ and $K_{2,3}$. Even if one member of the left hand side cuts trade with members of the other anti-bloc completely, perverse equilibria still may exist.
 
 If we consider the left-hand side anti-bloc to be sanctions busters, this illustrates that when there are many sanctions busters, convincing one of them to stop sanctions busting will not be pivotal, and will therefore have very little impact on the efficacy of the sanctions regime. Conversely, if we consider the left-hand side anti-bloc to be the senders and targets of sanctions, Figure \ref{fA} shows that one sender or target desisting from trading with sanctions busters is not pivotal in this instance, and perverse equilibrium may persist, even if they exit the international trade network completely. 

 In Figure \ref{fC}, instead of a metric that cuts off a single point, (the bottom left vertices of Figures 3 and 4) we instead add the discrete metric, where each point is uniformly separated from every other point.  This  corresponds to  every member of a trading network adding tariffs to trade with every other member, i.e. every state in the system resorting to isolationism.  It is noteworthy that as a network wide strategy for forcing Mossay-Tabuchi stability, this is actually a perfectly viable option. Furthermore, shocks to the cost of fuel could act similarly. Technological advancements facilitating lower cost trade, to the extent they are shared by all trading partners, may have the opposite effect. That is, moving the network from a Mossay-Tabuchi stable status quo to one that is not Mossay-Tabuchi stable. 
 
 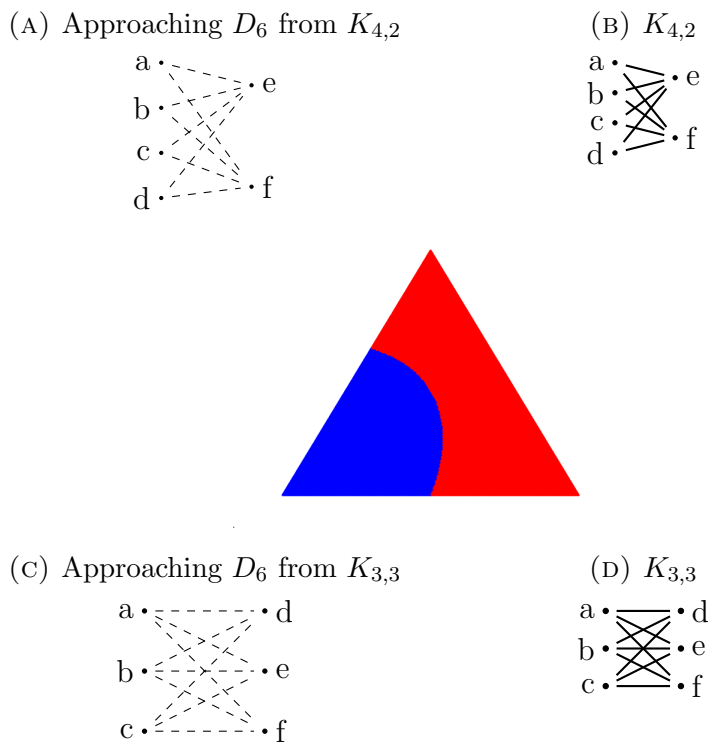
\begin{figure}[H]
   \begin{subfigure}[t]{0.45\linewidth}
       \caption{Approaching $D_{6}$ from $K_{4,2}$}
    \centering
       \begin{tikzpicture} [baseline = (n2.center), scale=.3]
        \node (n1) at (-2,2) {}; 
        \node (n2) at (-2,0) {};
        \node (n3) at (-2,-2) {};
        \node (n4) at (-2, -4) {};
        
        \node (n5) at (2,1) {};
        \node (n6) at (2,-3.5) {};
        
        \node[left] at (n1) {a};
        \node[left] at (n2) {b};
        \node[left] at (n3) {c}; 
        \node[left] at (n4) {d};
        \node[right] at (n5) {e};
         \node[right] at (n6) {f};
        
        \filldraw[black] (n1) circle(2 pt);
        \filldraw[black] (n2) circle(2 pt);
        \filldraw[black] (n3) circle(2 pt);
        \filldraw[black] (n4) circle(2 pt);
        \filldraw[black] (n5) circle(2 pt);
        \filldraw[black] (n6) circle(2 pt);
        
        \draw[black,dashed] (n1)--(n5);
        \draw[black,dashed] (n1)--(n6);

        \draw[black,dashed] (n2)--(n5);
        \draw[black,dashed] (n2)--(n6);

        \draw[black,dashed] (n3)--(n5);
        \draw[black,dashed] (n3)--(n6);

        \draw[black,dashed] (n4)--(n5);
        \draw[black,dashed] (n4)--(n6);

    \end{tikzpicture}
   \end{subfigure}
   \begin{subfigure}[t]{0.45\linewidth}
    \centering
    \caption{$K_{4,2}$}
    \centering
    \begin{tikzpicture} [baseline = (n2.center), scale=.4]
        \node (n1) at (-1,0) {}; 
        \node (n2) at (-1,-1) {};
        \node (n3) at (-1,-2) {};
        \node (n4) at (-1, -3) {};
        
        \node (n5) at (1,-0.5) {};
        \node (n6) at (1,-2.5) {};
        
        \node[left] at (n1) {a};
        \node[left] at (n2) {b};
        \node[left] at (n3) {c}; 
        \node[left] at (n4) {d};
        \node[right] at (n5) {e};
         \node[right] at (n6) {f};
        
        \filldraw[black] (n1) circle(2 pt);
        \filldraw[black] (n2) circle(2 pt);
        \filldraw[black] (n3) circle(2 pt);
        \filldraw[black] (n4) circle(2 pt);
        \filldraw[black] (n5) circle(2 pt);
        \filldraw[black] (n6) circle(2 pt);
        
        \draw[black,thick] (n1)--(n5);
        \draw[black,thick] (n1)--(n6);

        \draw[black,thick] (n2)--(n5);
        \draw[black,thick] (n2)--(n6);

        \draw[black,thick] (n3)--(n5);
        \draw[black,thick] (n3)--(n6);

        \draw[black,thick] (n4)--(n5);
        \draw[black,thick] (n4)--(n6);

    \end{tikzpicture}
   \end{subfigure}

   \begin{center}
   \begin{subfigure}[c]{0.40\linewidth}
    \includegraphics[width=\linewidth]{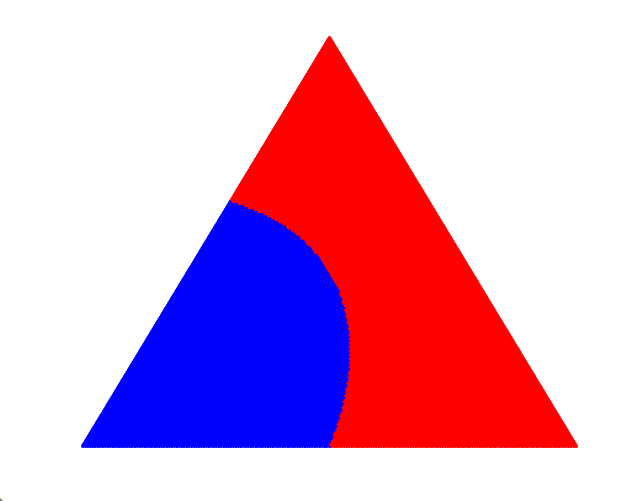}

   \end{subfigure}

   \end{center}

    \begin{subfigure}[t]{0.45\linewidth}
       \caption{Approaching $D_{6}$ from $K_{3,3}$}
    \centering
    \begin{tikzpicture} [baseline = (n2.center), scale=.4]
        \node (n1) at (-2,1) {}; 
        \node (n2) at (-2,-1) {};
        \node (n3) at (-2,-3) {};
        
        \node (n4) at (2, 1) {};
        \node (n5) at (2, -1) {};
        \node (n6) at (2,-3) {};
        
        \node[left] at (n1) {a};
        \node[left] at (n2) {b};
        \node[left] at (n3) {c}; 
        \node[right] at (n4) {d};
        \node[right] at (n5) {e};
         \node[right] at (n6) {f};
        
        \filldraw[black] (n1) circle(2 pt);
        \filldraw[black] (n2) circle(2 pt);
        \filldraw[black] (n3) circle(2 pt);
        \filldraw[black] (n4) circle(2 pt);
        \filldraw[black] (n5) circle(2 pt);
        \filldraw[black] (n6) circle(2 pt);
        
        \draw[black,dashed] (n1)--(n4);
        \draw[black,dashed] (n1)--(n5);
        \draw[black,dashed] (n1)--(n6);
        
        \draw[black,dashed] (n2)--(n4);
        \draw[black,dashed] (n2)--(n5);
        \draw[black,dashed] (n2)--(n6);
        
        \draw[black,dashed] (n3)--(n4);
        \draw[black,dashed] (n3)--(n5);
        \draw[black,dashed] (n3)--(n6);
        
    \end{tikzpicture}
   \end{subfigure}
     \begin{subfigure}[t]{0.45\linewidth}
    \centering
    \caption{$K_{3,3}$}
    \begin{tikzpicture} [baseline = (n2.center), scale=.5]
        \node (n1) at (-1,0) {}; 
        \node (n2) at (-1,-1) {};
        \node (n3) at (-1,-2) {};
        
        \node (n4) at (1, 0) {};
        \node (n5) at (1,-1) {};
        \node (n6) at (1,-2) {};
        
        \node[left] at (n1) {a};
        \node[left] at (n2) {b};
        \node[left] at (n3) {c}; 
        \node[right] at (n4) {d};
        \node[right] at (n5) {e};
         \node[right] at (n6) {f};
        
        \filldraw[black] (n1) circle(2 pt);
        \filldraw[black] (n2) circle(2 pt);
        \filldraw[black] (n3) circle(2 pt);
        \filldraw[black] (n4) circle(2 pt);
        \filldraw[black] (n5) circle(2 pt);
        \filldraw[black] (n6) circle(2 pt);
        
        \draw[black,thick] (n1)--(n4);
        \draw[black,thick] (n1)--(n5);
        \draw[black,thick] (n1)--(n6);
        
        \draw[black,thick] (n2)--(n4);
        \draw[black,thick] (n2)--(n5);
        \draw[black,thick] (n2)--(n6);
        
        \draw[black,thick] (n3)--(n4);
        \draw[black,thick] (n3)--(n5);
        \draw[black,thick] (n3)--(n6);
        
    \end{tikzpicture}
   \end{subfigure}
   
   \caption{\footnotesize{The impact of universal trade barriers on MT Stability. The vertices are $D_6$ (Bottom Left), $K_{4,2}$ (Top), $K_{3,3}$ (Bottom Right). $D_6$ is the discrete metric, in which all states in the network erect barriers to trade with all others. Interior points are linear combinations of the networks at the vertices. Neither $K_{4,2}$ nor $K_{3,3}$ are MT Stable. However, by universal introduction of barriers  to trade (approaching $D_6$) both can be made MT Stable. The diagrams A-D provide examples of the networks at the vertices.}} \label{fC}
    \end{figure}

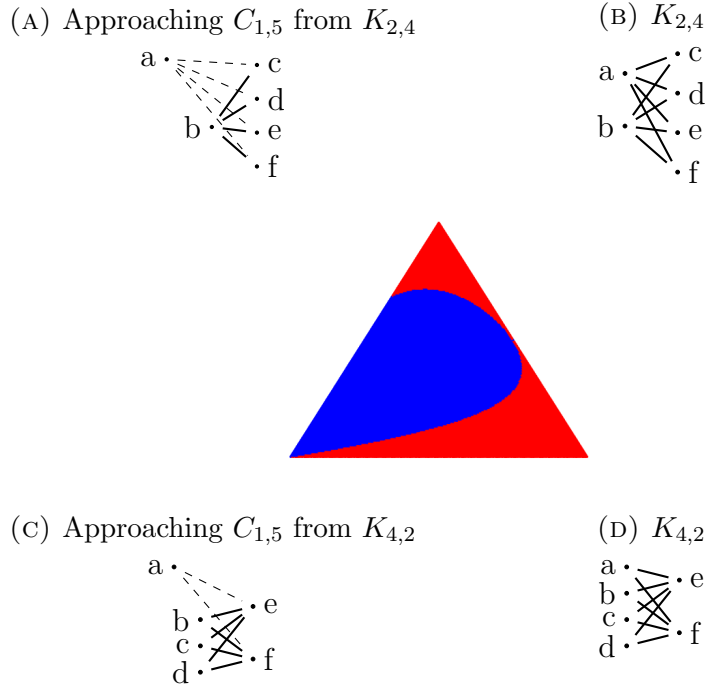
\begin{figure}[htbp] 
    \begin{subfigure}[]{0.45\linewidth}
    \centering
    \caption{Approaching $C_{1,5}$ from $K_{2,4}$}
    \centering
    \begin{tikzpicture} [baseline = (n2.center), scale=.3]
        \node (n1) at (-3,1.75) {}; 
        \node (n2) at (-1,-1.25) {};
        
        \node (n3) at (1, 1.5) {};
        \node (n4) at (1, 0) {};
        \node (n5) at (1, -1.5) {};
        \node (n6) at (1,-3) {};
        
        \node[left] at (n1) {a};
        \node[left] at (n2) {b};
        \node[right] at (n3) {c}; 
        \node[right] at (n4) {d};
        \node[right] at (n5) {e};
         \node[right] at (n6) {f};
        
        \filldraw[black] (n1) circle(2 pt);
        \filldraw[black] (n2) circle(2 pt);
        \filldraw[black] (n3) circle(2 pt);
        \filldraw[black] (n4) circle(2 pt);
        \filldraw[black] (n5) circle(2 pt);
        \filldraw[black] (n6) circle(2 pt);

         \draw[black,dashed] (n1)--(n3);
        \draw[black, dashed] (n1)--(n4);
        \draw[black,dashed] (n1)--(n5);
        \draw[black,dashed ] (n1)--(n6);

         \draw[black,thick] (n2)--(n3);
        \draw[black,thick] (n2)--(n4);
        \draw[black,thick] (n2)--(n5);
        \draw[black,thick] (n2)--(n6);
        
    \end{tikzpicture}
   \end{subfigure}
   \begin{subfigure}[]{0.45\linewidth}
    \centering
    \caption{$K_{2,4}$}
    \centering
    \begin{tikzpicture} [baseline = (n2.center), scale=.35]
        \node (n1) at (-1,0.75) {}; 
        \node (n2) at (-1,-1.25) {};
        
        \node (n3) at (1, 1.5) {};
        \node (n4) at (1, 0) {};
        \node (n5) at (1, -1.5) {};
        \node (n6) at (1,-3) {};
        
        \node[left] at (n1) {a};
        \node[left] at (n2) {b};
        \node[right] at (n3) {c}; 
        \node[right] at (n4) {d};
        \node[right] at (n5) {e};
         \node[right] at (n6) {f};
        
        \filldraw[black] (n1) circle(2 pt);
        \filldraw[black] (n2) circle(2 pt);
        \filldraw[black] (n3) circle(2 pt);
        \filldraw[black] (n4) circle(2 pt);
        \filldraw[black] (n5) circle(2 pt);
        \filldraw[black] (n6) circle(2 pt);

         \draw[black,thick] (n1)--(n3);
        \draw[black,thick] (n1)--(n4);
        \draw[black,thick] (n1)--(n5);
        \draw[black,thick] (n1)--(n6);

         \draw[black,thick] (n2)--(n3);
        \draw[black,thick] (n2)--(n4);
        \draw[black,thick] (n2)--(n5);
        \draw[black,thick] (n2)--(n6);
        
    \end{tikzpicture}
   \end{subfigure}

   \begin{center}
        \begin{subfigure}[c]{0.40\linewidth}
       \includegraphics[width=\linewidth]{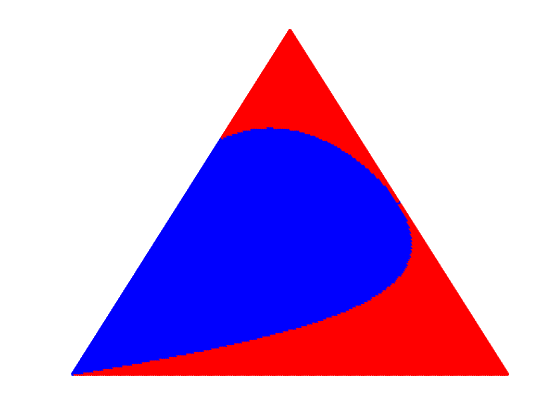}
   \end{subfigure}
   \end{center}

    \begin{subfigure}[t]{0.45\linewidth}
       \caption{Approaching $C_{1,5}$ from $K_{4,2}$}
    \centering
    \begin{tikzpicture} [baseline = (n2.center), scale=.35]
        \node (n1) at (-2,1) {}; 
        \node (n2) at (-1,-1) {};
        \node (n3) at (-1,-2) {};
        \node (n4) at (-1, -3) {};
        
        \node (n5) at (1,-0.5) {};
        \node (n6) at (1,-2.5) {};
        
        \node[left] at (n1) {a};
        \node[left] at (n2) {b};
        \node[left] at (n3) {c}; 
        \node[left] at (n4) {d};
        \node[right] at (n5) {e};
         \node[right] at (n6) {f};
        
        \filldraw[black] (n1) circle(2 pt);
        \filldraw[black] (n2) circle(2 pt);
        \filldraw[black] (n3) circle(2 pt);
        \filldraw[black] (n4) circle(2 pt);
        \filldraw[black] (n5) circle(2 pt);
        \filldraw[black] (n6) circle(2 pt);

        \draw[black,dashed] (n1)--(n5);
        \draw[black,dashed] (n1)--(n6);

        \draw[black,thick] (n2)--(n5);
        \draw[black,thick] (n2)--(n6);

        \draw[black,thick] (n3)--(n5);
        \draw[black,thick] (n3)--(n6);
        
          \draw[black,thick] (n4)--(n5);
        \draw[black,thick] (n4)--(n6);

    \end{tikzpicture}
   \end{subfigure}
   \begin{subfigure}[t]{0.45\linewidth}
    \centering
    \caption{$K_{4,2}$}
    \centering
    \begin{tikzpicture} [baseline = (n2.center), scale=.35]
        \node (n1) at (-1,0) {}; 
        \node (n2) at (-1,-1) {};
        \node (n3) at (-1,-2) {};
        \node (n4) at (-1, -3) {};
        
        \node (n5) at (1,-0.5) {};
        \node (n6) at (1,-2.5) {};
        
        \node[left] at (n1) {a};
        \node[left] at (n2) {b};
        \node[left] at (n3) {c}; 
        \node[left] at (n4) {d};
        \node[right] at (n5) {e};
         \node[right] at (n6) {f};
        
        \filldraw[black] (n1) circle(2 pt);
        \filldraw[black] (n2) circle(2 pt);
        \filldraw[black] (n3) circle(2 pt);
        \filldraw[black] (n4) circle(2 pt);
        \filldraw[black] (n5) circle(2 pt);
        \filldraw[black] (n6) circle(2 pt);
        
        \draw[black,thick] (n1)--(n5);
        \draw[black,thick] (n1)--(n6);

        \draw[black,thick] (n2)--(n5);
        \draw[black,thick] (n2)--(n6);

        \draw[black,thick] (n3)--(n5);
        \draw[black,thick] (n3)--(n6);

        \draw[black,thick] (n4)--(n5);
        \draw[black,thick] (n4)--(n6);

    \end{tikzpicture}
   \end{subfigure}
   \caption{\footnotesize{The impact of one state leaving a small bloc versus one state leaving a large bloc. The vertices are $C_{1,5}$ (Bottom Left), $K_{2,4}$ (Top), $K_{4,2}$ (Bottom Right). $C_{1,5}$ is the cut metric in which one member of the left hand side antibloc leaves the trading network. Interior points are linear combinations of the networks at the vertices. Neither $K_{2,4}$ nor $K_{4,2}$ are MT Stable. A member leaving from the antibloc of size 2 induces MT stability, while a member leaving from the antibloc of size 4 does not. }}\label{fG}
    \end{figure}

  Figure \ref{fG} illustrates that where a cut is placed is important.  While $K_{4,2}$ and $K_{2,4}$ are essentially the same up to some rearrangement of a picture it is worth noting that \textit{where} distance is added is vital.  Adding $C_{1,5}$ to $K_{2,4}$, in essence, breaks up the small anti-block leaving a copy of $K_{4,1}$.  However, adding $C_{1,5}$ to $K_{4,2}$ breaks off one of the members of the large anti-bloc leaving a copy of $K_{3,2}$.  As a result, the Mossay-Tabuchi stable region of this subcone is highly asymmetric. As the figure shows, a unilaterally increasing barriers to trade pivotal to inducing Mossay-Tabuchi stability when they are part of the small group of size two, and does not induce Mossay-Tabuchi stability when they are part of the large group of size four. 

As these examples show, considering the geometry of trade freeness, i.e. the geometry of inconvenience, is central to understanding the impact of policy interventions. We next turn to a discussion of implications for future research in this area. 

\FloatBarrier
\section{Conclusion}

Above, we show that the network of trade ties may form a geometry of inconvenience in which multiple perverse equilibria are possible. In particular, we identify bipartite trade networks, those which have at least two anti-blocs in which there is relatively lower cost across blocs and relatively higher cost trade within them as an example of one such geometry. We consider how such bipartite trade networks may form due to economic sanctions. Next we discussed the impact of a variety of unilateral and multilateral policy interventions. In conclusion, we discuss implications for attitudes toward globalization, other ways in which geometries of inconvenience may arise and their implications, and avenues for future research. 

While the model presented above is not dynamic, we offer some tentative  implications of the analysis above for attitudes toward globalization.  In particular, in partitioned systems, more extensive trade in the trade ties in a bipartite trade network may decrease general welfare due to increases in rent seeking, in addition to distributional consequences. This suggests a logic distinct from the conventional winners and losers from trade and imperfect redistribution mechanism suggested for for globalization backlash \citep{milner_1999}.  Rather than redistribution, losers due to trade network based rent-seeking may be placated by establishing new trade ties to undermine the status quo bipartite structure, but at the same time may oppose more extensive trade with status quo trade partners, perhaps favoring measures that depress trade, such as border walls \citep{carter2020}. 

In addition, future research should investigate the historical implications of geometries of inconvenience. For instance, mercantilist trading networks may be better understood using the framework we introduce here. Mercantilism was practiced during colonization to various degrees, limiting both home country and colony trading partners. Those colonies which acted as an intermediary often thrived. Those that were forced to use the colonizing country as an intermediary, and forgo trade with nearby trading partners due to imperial control floundered, while their colonizers reaped not only great profits, but also political rents, due to the deliberate structure of international trade \citep{zahedieh_2010, ekelund1981mercantilism, jones_1996, rommelse_2010}. 

This example helps illustrate the differing implications of the extensive vs intensive margins of trade, and the structures they create, that we identify. On the extensive margin, certain trade ties were restricted by colonizers, often limiting colonies to only trade with them, however on the intensive margin, colonizing states pushed for greater volumes of trade. This comports with our findings above regarding rent-seeking and the volume of trade in bipartite trade structures. Similarly, the geometry of inconvenience we identify helps us better understand the true welfare consequences of more free international trade within a partitioned  system.     

The framework introduced here may also be applied to better understand the role of geography in international trade \citep{GERVAIS2019331}. For instance, landlocked states may also form examples of the bipartite family.  A landlocked state, to reach global markets, must conduct trade through transit countries, some of which are also isolated. They often have multiple options of transit countries, with limited trade between the intermediaries themselves. Thus, again, we see a bipartite structure of international trade. Examples of perverse equilibria in this context are numerous \citep{arvis_2010}. 

In each of these examples, we see a new role of international structure in international trade.  Structural approaches have typically focused on the distribution of power or economic size. Here we have instead proposed a novel mechanism, the geometry of trade freeness, and shown that this factor is central to understanding economic outcomes and political goals. 

The results presented here also complement efforts to better understand the role of networks in international trade policy \citep{farrell2019, joshi2016, kinne2012, manger2012, cranmer2014}. The analysis presented here moves beyond the standard approaches which tend to focus on network centrality and density \citep{joshi2023}. We present a novel, policy-relevant, geometry of inconvenience to shed new light on the impacts of economic sanctions.

\bibliography{references}
\bibliographystyle{apalike}

\end{document}